\documentclass[12pt]{article}
\usepackage{amsmath, amsfonts, amsthm, amssymb, verbatim}
\usepackage{graphicx}
\usepackage[mathcal]{euscript}
\usepackage{amstext}
\usepackage{amssymb,mathrsfs}
\usepackage{bbm}
\usepackage[latin1]{inputenc}
\newcommand{\R}{\mathbb R}

  \newcommand{\E}{\mathbb E}

\newcommand{\PP}{\mathbb P}

\newcommand{\calR} {\ensuremath {\mathcal{R}}}
\newcommand{\calF} {\ensuremath {\mathcal{F}}}

\newcommand \loc    {\text{loc}}

\newcommand{\dive}{{\rm div}}

\newcommand{\I}{\int}

\newtheorem{theorem}{Theorem}[section]

 \newtheorem{remark}[theorem]{Remark}
\newtheorem{lemma}[theorem]{Lemma}

\newtheorem{definition}[theorem]{Definition}
\newtheorem{hypothesis}[theorem]{Hypothesis}

\begin{document}

\title{    Stochastic continuity  equation with non-smooth velocity. }

\author{David A.C. Mollinedo \footnote{Departamento Acadêmico de Matem\'{a}tica, Universidade Tecnológica Federal do Paraná, Brazil. E-mail:  {\sl  davida@utfpr.edu.br}
},
Christian Olivera\footnote{Departamento de Matem\'{a}tica, Universidade Estadual de Campinas, Brazil.
E-mail:  {\sl  colivera@ime.unicamp.br}.
}}

\date{}

\maketitle


\vspace{0.3cm} \noindent {\bf MSC2010 subject classification:} 60H15, 
 35R60, 
 35F10, 
 60H30. 


%
\begin{abstract}
In this article we study  the existence and uniqueness of solutions of the 
stochastic  continuity  equation with irregular coefficients.  

 \end{abstract}
%
\maketitle

%

\section {Introduction} \label{Intro}

Several physical phenomena arising in fluid dynamics and kinetic equations can be modeled by the
 continuity/ transport equation,
\begin{equation}\label{trasports}
    \partial_t u(t, x) + div ( b(t,x)  u(t,x) ) = 0 \, ,
\end{equation}

\noindent where $u$ is the physical
quantity that evolves in time. Such quantities are the vorticity of a fluid, or the density of a collection of
particles advected by a velocity field which is highly irregular, in the sense that it has a derivative given by
a distribution and a nonlinear dependence on the solution u. For application in the fluid dynamics
see Lions' books \cite{lion1}, \cite{lion2} and for applications in the domain of conservation laws see Dafermos' book \cite{Dafermos}.

\noindent Recently research activity has been devoted to study continuity equations with rough coefficients, showing a well-posedness result.
 We put focus in the uniqueness issue. Di Perna and Lions \cite{DL} have introduced the notion of renormalized solution to this equation: it is a solution such that

\begin{equation}\label{renord}
    \partial_t \beta (u(t, x) )+  div (b(t,x) \cdot  \beta (u(t,x))  = 0 .
\end{equation}

\noindent for any suitable non-linearity $\beta$. Notice that (\ref{renord}) holds for smooth solutions, by an immediate application of the chain-rule. The renormalization property asserts that nonlinear compositions of the solution are again solutions, or alternatively that the chain-rule holds in this weak
context. The overall result which motivates this definition is that, if the renormalization property holds, then solutions of (\ref{trasports}) are unique and stable.

\medskip

In the case when $b$ has $W^{1,1}$ spatial regularity (together with a condition of boundedness on the divergence) the commutator lemma between smoothing convolution and weak solution can be proved and, as a consequence, all $L^{\infty}$-weak solutions are renormalized. The theory has been generalized by L. Ambrosio \cite{ambrisio} to the case of only $BV$ regularity for b instead of $W^{1,1}$. In the case of two-dimensional vector-field, we also refer to the work of F. Bouchut and L. Desvillettes \cite{Bouchut2} that treated the case of divergence free vector-field with continuous coefficient, and to \cite{MH} in which this result is extended to vector-field with $L_{loc}^{2}$ coefficients with a condition of regularity on the direction of the vector-field. We refer the readers to two excellent summaries in 
\cite{ambrisio2} and \cite{lellis}.

\medskip

In recent years, much attention has been devoted to extensions of this theory under
random perturbations of the drift vector field, namely considering the following stochastic linear transport/continuity equation
\begin{equation}\label{trasport}
 \left \{
\begin{aligned}
    &\partial_t u(t, x) + Div  \big( ( b(t, x) + \frac{d B_{t}}{dt}) \cdot  u(t, x)  \big )= 0 \, ,
    \\[5pt]
    &u|_{t=0}=  u_{0} \, .
\end{aligned}
\right .
\end{equation}
Here, $(t,x) \in [0,T] \times \R^d$, $\omega \in \Omega$ is an element of the probability space $(\Omega, \PP, \calF)$, $b:\R_+ \times \R^d \to \R^d$ is a given vector field and $B_{t} = (B_{t}^{1},...,B _{t}^{d} )$ is a standard Brownian motion in $\mathbb{R}^{d}$. The stochastic integration is to be understood in the Stratonovich sense.

\medskip

A very interesting situation is when the stochastic problem is better behaved than the deterministic one. A first result in this direction was
given by F. Flandoli, M. Gubinelli and  E. Priola
in \cite{FGP2}, where they obtained well-posedness of the stochastic problem for an H\"older continuous drift term, with some integrability conditions on the divergence. Their approach is based on a careful analysis of the characteristics. Using a similar approach, E.Fedrizi and F. Flandoli in \cite{Fre1} obtained a well-posedness result in the class $W_{loc}^{1,p}$-solution under only some integrability conditions on the drift, with no assumption on the divergence, but for fairly regular initial conditions. There, it is only assumed that
\begin{equation}\label{LPSC}
    \begin{aligned}
    &b\in L^{q}\big( [0,T] ; L^{p}(\mathbb{R}^{d}) \big) \, , \\[5pt]
 \mathrm{for}  \qquad &\  p,q \in [2,\infty) \, , \qquad   \qquad  \frac{d}{p} + \frac{2}{q} < 1 \, .
\end{aligned}
\end{equation}
In fact, this condition (with local integrability) was first considered by Krylov and R\"{o}ckner in \cite{Krylov}, where they proved the existence and uniqueness of strong solutions for the SDE
\begin{equation}\label{itoass}
X_{s,t}(x)= x + \int_{s}^{t}   b(r, X_{s,r}(x)) \ dr  +  B_{t}-B_{s} \, ,
\end{equation}
such that
$$
 \mathbb{P}\Big( \int_0^T |b(t,X_t)|^2 \ dt< \infty \Big)= 1 \, .
$$
Similarly, we may consider for convenience the inverse  $Y_{s,t}:=X_{s,t}^{-1}$, which satisfies the following backward stochastic
differential equations,
\begin{equation}\label{itoassBac}
Y_{s,t}= y - \int_{s}^{t}   b(r, Y_{r,t}) \ dr  - (B_{t}-B_{s}),
\end{equation}
for $0\leq s\leq t$.

\medskip

The well-posedness of the Cauchy problem \eqref{trasport} under condition \eqref{LPSC} for measurable initial condition was also considered in \cite{NO} and \cite{NO2}. In \cite{Beck}, using a technique based on the regularizing effect observed on expected values of moments of the solution, well-posedness of \eqref{trasport} was obtained also for the limit cases of $p,q=\infty$ or when the inequality in \eqref{LPSC} becomes an equality. The uniqueness result in that paper are valid for solutions in weighted spaces.

\medskip

We mention that other approaches have also been used to study stochastic linear transport/continuity equations. For example, M. Maurelli in \cite{Maurelli} employed the Wiener chaos decomposition to deal with a weakly differentiable drift, in \cite{MNP14},
 S.A. Mohammed, T.K. Nilssen, F.N. Proske used  Malliavin calculus which allows to deal with just a bounded drift, and in \cite{Fre2} the authors introduced a new class of solutions. We would also like to mention the generalizations to transport-diffusion equations and the associated stochastic differential equations by A. Figalli \cite{Figalli} and X. Zhang \cite{Zhang}.

\medskip

The main issue of this paper is to prove uniqueness of  $L^{2}$-weak solutions for one-dimensional stochastic continuity
equation (\ref{trasport}) with unbounded measurable drift without assumptions on the divergence. More precisely, we assume that $b$ satisfies

\begin{equation*}
  |b(x)|\leq k (1 +|x|).
\end{equation*}

\medskip The proof is based on the fact that one primitive $V$ is regular and verifies
the transport equation
\begin{equation}\label{aux}
    \partial_t V(t, x) +  ( b(t, x) + \frac{d B_{t}}{dt}) \cdot  \nabla V(t, x)  = 0 \, .
\end{equation}

Then using a modified version of the commutator Lemma and the characteristic systems associated to the SPDE (\ref{aux}) we shall show that $V=0$ with initial condition equal to zero, which implies that
$u=0$.
\medskip

Other issue in this paper is to give a well-posedness result for solutions in the Sobolev spaces $H^{1}(\mathbb{R}^{d})$ under condition (\ref{LPSC}) with divergence equal to zero. In particular this result implies the persistence of the regularity for initial conditions $u_0\in H^{1}(\mathbb{R}^{d})$. The proof is based in the Commutator Lemma given by C. Le Bris and P. L. Lions in \cite{BrisLion}
 for functions with Sobolev regularity. This new result shows the uniqueness in the class of $H^{1}$-solutions, not covered in the previous works (see \cite{Beck,Fre1,NO,NO2}) under this condition.  

Throughout of this paper, we fix a stochastic basis with a $d$-dimensional Brownian motion $\big( \Omega, \mathcal{F}, \{\mathcal{F}_t: t \in [0,T] \}, \mathbb{P}, (B_{t}) \big)$.


\section{$L^{2}$- Solutions.}

In this section we assume the following hypothesis:

\begin{hypothesis}\label{hyp1}
The vector field $b$ satisfies
\begin{equation}\label{cond3-1}
  |b(x)|\leq k (1 +|x|)\,,
\end{equation}
and the initial condition holds
\begin{equation}\label{weight}
  u_0 \in  L^2(\R, w\,dx)
\end{equation}
where $w$ is the weight defined by $w(x)=e^{ 2  k_2 x^2}$ with $k_2=2(k+99 T k^2)$.
\end{hypothesis}

Now,  we denote by $b^{\epsilon}$ the standard mollification of $b$, and let   $X_t^{\epsilon}$  be  the associated flow given by  the SDE (\ref{itoass}) replacing  $b$ by $b^{\epsilon}$.  Similarly, we consider $Y^\epsilon_{t}$, which satisfies the backward SDE \eqref{itoassBac}. We also recall   the important results  in  \cite{nilssen} (see appendix ) : let $X_{t}^{\epsilon}$ be the corresponding stochastic flows, then for all $p\geq 1$   there are constants $C_1=C_1(k,p,T)$ and $C_2(k,p,T)$  such that
\begin{equation}\label{est3}
  \mathbb{E}[|\partial_x X_{t}^{\epsilon}(x)|^p]\leq C_1  t^{-\frac{1}{2}}  e^{C_2   x^{2}},
\end{equation}
the same results is valid for the backward flow $Y_{t}^{\epsilon}$  since it is  solution of the same SDE driven by the drifts $-b^{\epsilon}$.
We denote $\mu=(1+ |x|)^{2}$.

\subsection{Definition of solutions}
\begin{definition}\label{defisoluH}
A stochastic process $u\in   L^2(\Omega\times [0,T]\times \R, \mu dx) $  is called a  $L^{2}$- weak solution of the Cauchy problem \eqref{trasport} when: For any $\varphi \in C_0^{\infty}(\R)$, the real valued process $\int  u(t,x)\varphi(x)  dx$ has a continuous modification which is an $\mathcal{F}_{t}$-semimartingale, and for all $t \in [0,T]$, we have $\mathbb{P}$-almost surely
\begin{equation} \label{DISTINTSTR}
\begin{aligned}
    \int_{\R} u(t,x) \varphi(x) dx = &\int_{\R} u_{0}(x) \varphi(x) \ dx
	+ \int_{0}^{t} \!\! \int_{\R}   u(s,x)   \, b (x) \partial_x \varphi(x) \ dx ds
\\[5pt]
    & + \int_{0}^{t} \!\! \int_{\R}   u(s,x) \ \partial_x \varphi(x) \ dx \, {\circ}{dB_s} \, .
\end{aligned}
\end{equation}
\end{definition}

\begin{remark}\label{lemmaito}
Using the same idea as in Lemma 13 \cite{FGP2}, one can write the problem (\ref{trasport}) in It\^o form as follows, a  stochastic process $u\in   L^2(\Omega\times [0,T]\times \R, \mu dx) $ is  a $ L^{2}$- weak solution  of the SPDE (\ref{trasport}) iff for every test function $\varphi \in C_{0}^{\infty}(\mathbb{R})$, the process $\int u(t, x)\varphi(x) dx$ has a continuous modification which is a $\mathcal{F}_{t}$-semimartingale and satisfies the following It\^o's formulation

\[
\begin{aligned}
    \int_{\R} u(t,x) \varphi(x) dx = &\int_{\R} u_{0}(x) \varphi(x) \ dx
	+ \int_{0}^{t} \!\! \int_{\R}   u(s,x)   \, b (x) \partial_x \varphi(x) \ dx ds
\\[5pt]
    & + \int_{0}^{t} \!\! \int_{\R}   u(s,x) \ \partial_x \varphi(x) \ dx \, dB_s \,  + \frac{1}{2} \int_{0}^{t} \!\! \int_{\R}   u(s,x) \ \partial_x^{2} \varphi(x) \ dx \, ds.
\end{aligned}
\]

\end{remark}

\subsection{Existence.}

We shall here prove existence of  solutions  under hypothesis \ref{hyp1}.
\begin{lemma}
Assume that hypothesis \ref{hyp1} holds. Then there exists $L^2$-weak solution of the Cauchy problem \eqref{trasport}.
\end{lemma}

\begin{proof}

{\it Step 1: Regularization.}

 Let $\{\rho_\varepsilon\}_\varepsilon$ be a family of standard symmetric mollifiers and $\eta$ a nonnegative smooth cut-off function supported on the ball of radius 2 and such that $\eta=1$ on the ball of radius 1. Now, for every $\varepsilon>0$, we introduce the rescaled functions $\eta_\varepsilon (\cdot) = \eta(\varepsilon \cdot)$. Thus, we  define the family of  regularized coefficients given by
$$b^{\epsilon}(x) = \eta_\varepsilon(x) ( b \ast \rho_\varepsilon  (x)) $$
and
$$u_0^\varepsilon (x) = \eta_\varepsilon(x) \big(  u_0 \ast \rho_\varepsilon (x)  \big) \,.$$
Clearly we observe that, for every  $\varepsilon>0$, any element $b^{\varepsilon}$, $u_0^\varepsilon$ are smooth (in space) and have  compactly supported with bounded derivatives of all orders.  We observe that to study the stochastic continuity equation (SCE) \eqref{trasport} is equivalent to study the stochastic transport equation given by (regularized version):
\begin{equation}\label{STE-reg}
 \left \{
\begin{aligned}
    &d u^\varepsilon (t, x) +  \nabla u^\varepsilon (t, x)  \cdot \big( b^\varepsilon (x)  dt +
 \circ d B_{t} \big) +\dive b^{\varepsilon}(x) \,u^\varepsilon (t,x) dt = 0\, ,
    \\[5pt]
    &u^\varepsilon \big|_{t=0}=  u_{0}^\varepsilon
\end{aligned}
\right .
\end{equation}
Following the classical theory of H. Kunita \cite[Theorem 6.1.9]{Ku} we obtain that

\[
u^{\varepsilon}(t,x) =  u_{0}^{\varepsilon} (\psi_t^{\varepsilon}(x)) \exp\bigg\{-\I_0^t \dive b^{\varepsilon}(\phi_s^{\varepsilon}(\psi_t^{\varepsilon}(x))) ds \bigg\}
\]
is the unique solution to the regularized equation \eqref{STE-reg}, where $\phi_t^{\varepsilon}$ is the flow associated  to the following stochastic differential equation (SDE):
\begin{equation*}
d X_t = b^\varepsilon (X_t) \, dt + d B_t \, ,  \hspace{1cm}   X_0 = x \,.
\end{equation*}
and $\psi_t^{\varepsilon}$ is the inverse of $\phi_t^{\varepsilon}$.

\bigskip

{\it Step 2: Boundedness.} Making the change of variables $y=\psi_t^{\varepsilon}(x)=(\phi_t^{\varepsilon}(x))^{-1}$ we have that
\begin{align*}
\int_{\R} \E[|u^{\varepsilon}(t,x)|^2]\, (1+ |x|)^{2} dx & =\int_{\Omega}\int_{\R} |u_{0}^{\varepsilon} (y)|^2 \exp\bigg\{-2\int_0^t \dive b^{\varepsilon}(\phi_s^{\varepsilon}(y)) ds \bigg\} \times \\[5pt]
& \quad\quad\quad\quad\quad\times\frac{d \phi_t^{\varepsilon}(y)}{dy} \ (1+ |\phi_t^{\varepsilon}(y)|)^{2}  \, dy \mathbb{P}(d\omega).
\end{align*}
Now, if we do a minor modification in the proof of the Lemma 3.6 of \cite{nilssen} (see appendix) we obtain that there are constants $k_1=k_1(k,T)$ and $k_2=2(k+99 T k^2)$ such that
\begin{align}\label{eq0}
\mathbb{E}\bigg[\bigg|\frac{d}{dx}\phi^{\varepsilon}_t(x)\bigg|^{-2}\bigg] = \mathbb{E}\bigg[   \exp\bigg\{-2\int_0^t \dive b^{\varepsilon}(\phi_s^{\varepsilon}(x)) ds \bigg\}    \bigg]               \leq  k_1  t^{-3/8} e^{k_2 x^2}  \,.
\end{align}
We also observe that
\begin{align}\label{eq2}
\mathbb{E}\bigg[ | \phi_t^{\varepsilon}(x)  |^{4}  \bigg]\leq C (|x|^{4} + T^{4})
\end{align}
Then we obtain 
\[
\mathbb{E} \bigg[ \bigg|\frac{d}{dx}\phi^{\varepsilon}_t(x)\bigg|^{-1} (1+ |\phi_t^{\varepsilon}(x)|)^{2}  \bigg]
\]

\[
 \leq
C \bigg(  \mathbb{E} \bigg|\frac{d}{dx}\phi^{\varepsilon}_t(x)\bigg|^{-2}  +   \mathbb{E}  \bigg|  (1+ |\phi_t^{\varepsilon}(x)|)^{4}
  \bigg|       \bigg)
 \leq C(k_1  t^{-3/8} e^{k_2 x^2} + T^4+x^4).
\]
Thus we deduce 
\begin{align}\label{uno}
\int_{\R} & \E[|u^{\varepsilon}(t,x)|^2](1+|x|)^{2}\, dx  \leq \int_{\R} |u_{0}^{\varepsilon} (y)|^2 \E\bigg[\bigg|\frac{d \phi_s^{\varepsilon}(y)}{dy}\bigg|^{-1} (1+ |\phi_t^{\varepsilon}(y)|)^{2} \bigg]\, dy \nonumber\\[5pt]
& \leq C \int_{\R} |u_{0}^{\varepsilon} (y)|^2     \big( k_1  t^{-3/8} e^{k_2 x^2} + T^4+y^4 \big)   \, dy \nonumber\\[5pt]
& \leq  Ck_1 t^{-3/8} \int_{\R} |u_{0}^{\varepsilon} (y)|^2 e^{k_2 y^2}  \ dy + C\int_{\R} |u_{0}^{\varepsilon} (y)|^2    e^{k_2 y^2}  \, dy  .
\end{align}
We observe that 
\begin{align}\label{dos}
\int_{\R} |u_{0}^{\varepsilon} (y)|^2 e^{k_2 y^2} dy & \leq  \int_{\R} \bigg[ e^{k_2 y^2} \bigg(\int_{\R}\rho_{\varepsilon}(y-x)|u_0(x)|^2dx\bigg)\bigg] dy \nonumber\\[5pt]
& = \I_{\R} \bigg[ |u_0(x)|^2 \bigg(\int_{B(x,\varepsilon)}\rho_{\varepsilon}(y-x)e^{k_2 y^2}dy\bigg)\bigg] dx \nonumber\\[5pt]
& = \I_{\R} \bigg[ |u_0(x)|^2 \bigg(\int_{B(0,\varepsilon)}\rho_{\varepsilon}(u)e^{k_2 (x+u)^2}du\bigg)\bigg] dx \nonumber\\[5pt]
& \leq \I_{\R} \bigg[ |u_0(x)|^2 e^{2k_2 x^2}\bigg(\int_{B(0,\varepsilon)}\rho_{\varepsilon}(u)e^{2 k_2u^2}du\bigg)\bigg] dx  \nonumber\\[5pt]
& \leq C \|u_0\|^2_{L^2(\R, wdx)} .
\end{align}
From (\ref{uno}) and (\ref{dos}) we conclude that
\begin{align*}
\|u^{\varepsilon}\|^2_{L^2(\Omega\times [0,T]\times \R, \mu dx)} \leq C(k,T) \|u_0\|^2_{L^2(\R, wdx)}\,.
\end{align*}
Therefore, the sequence $\{u^{\varepsilon}\}_{\varepsilon>0}$ is bounded in $L^2(\Omega\times [0,T]\times \R, \mu dx )$. Then  there exists a convergent subsequence, which we denote also by $u^{\varepsilon}$, such that converge weakly in $L^2(\Omega\times [0,T]\times \R, \mu dx)$ to some process $u\in L^2(\Omega\times [0,T]\times \R, \mu dx)$ .

\bigskip
{\it Step 3: Passing to the Limit.}
Now, if $u^{\varepsilon}$ is a solution of \eqref{STE-reg}, it is also a weak solution, that is, for any test function $\varphi\in C_0^{\infty}(\R)$, $u^{\varepsilon}$ satisfies (written in the Itô form):
\begin{align*}
\int_{\R} u^{\varepsilon}(t,x) \varphi(x) dx = &\int_{\R} u^{\varepsilon}_{0}(x) \varphi(x) \ dx + \int_{0}^{t} \!\! \int_{\R}   u^{\varepsilon}(s,x)   \, b^{\varepsilon} (x) \partial_x \varphi(x) \ dx ds \\
    & + \int_{0}^{t} \!\! \int_{\R}   u^{\varepsilon}(s,x) \ \partial_x \varphi(x) \ dx \, dB_s \,  + \frac{1}{2} \int_{0}^{t} \!\! \int_{\R}   u^{\varepsilon}(s,x) \ \partial_x^{2} \varphi(x) \ dx \, ds\,.
\end{align*}
Thus, for prove existence of the SCE \eqref{trasport} is enough to pass to the limit in the above equation along the convergent subsequence found. This is made through of the same arguments of \cite[theorem 15]{FGP2}.

\end{proof}


\subsection{Uniqueness.}

\begin{theorem}\label{uni2}
Under the conditions of hypothesis \ref{hyp1}, uniqueness holds for  $L^{2}$- weak solutions of the Cauchy problem \eqref{trasport} in the following sense: if $u,v$ are $L^{2}$- weak solutions with the same initial data $u_{0}\in  L^2(\R, w\,dx)$, then  $u= v$ almost everywhere in $ \Omega  \times [0,T] \times \R$.
\end{theorem}

\begin{proof}

{\it Step 0: Set of solutions.} Remark that the set of  $L^{2}$- weak solutions is a linear subspace of $L^{2}( \Omega\times[0, T]\times \mathbb{R} , \mu dx)$, because the stochastic continuity equation is linear, and the regularity conditions is a linear constraint. Therefore, it is enough to show that a $L^{2}$- weak solution $u$ with initial condition $u_0= 0$ vanishes identically.

\bigskip

{\it Step 1:  Primitive of the solution.} We define $V(t,x)=\int_{-\infty}^{x} u(t,y) \ dy $ and we observe that $\partial_x V(t,x)$  belong to $L^{2}( \Omega\times[0, T]\times \mathbb{R},\mu dx )$. We consider a  nonnegative smooth cut-off function $\eta$ supported on the ball of radius 2 and such that  $\eta=1$ on the ball of radius 1. For any $R>0$, we introduce the rescaled functions $\eta_R (\cdot) =  \eta(\frac{.}{R})$.
Let be $\varphi\in C_0^{\infty}(\R)$, we observe that

\[
 \int_{\R} V(t,x) \varphi(x) \eta_R (x)  dx = - \int_{\R} u(t,x)   \theta(x) \eta_R (x)  dx
-\int_{\R} V(t,x)   \theta(x) \partial_x \eta_R (x)  dx\,,
\]

\noindent where $\theta(x) =\int_{-\infty}^{x}  \varphi(y)   \ dy$. By  definition of the solution $u$, taking as test function $ \theta(x) \eta_R (x)$ we have that $V(t,x)$  verifies

\begin{align}\label{DISTINTSTRTR}
    \int_{\R} & V(t,x) \ \eta_R (x) \varphi(x) dx = - \int_{0}^{t} \!\! \int_{\R}   \partial_x V(s,x)   \, b (x) \eta_R (x) \varphi(x) \ dx ds \nonumber\\[5pt]
    & - \int_{0}^{t} \!\! \int_{\R}   \partial_x V(s,x) \  \eta_R (x) \varphi(x) \ dx \, {\circ}{dB_s} - \int_{0}^{t} \!\! \int_{\R}   \partial_x V(s,x)   \, b (x) \partial_x \eta_R (x) \theta(x)  \ dx ds\nonumber\\[5pt]
    &- \int_{0}^{t} \!\! \int_{\R}   \partial_x V(s,x) \  \partial_x \eta_R (x) \theta(x)  \ dx \, {\circ}{dB_s}-\int_{\R} V(t,x)   \theta(x) \partial_x \eta_R (x)  dx.
\end{align}

Taking the limit as $R\rightarrow \infty$ we obtain

\begin{equation} \label{DISTINTSTRT}
\begin{aligned}
    \int_{\R} V(t,x) \varphi(x) dx = \\[5pt]
	- \int_{0}^{t} \!\! \int_{\R}   \partial_x V(s,x)   \, b (x)  \varphi(x) \ dx ds
- \int_{0}^{t} \!\! \int_{\R}   \partial_x V(s,x) \   \varphi(x) \ dx \, {\circ}{dB_s} .
\end{aligned}
\end{equation}

\bigskip

{\it Step 2: Smoothing.}
Let $\{\rho_{\varepsilon}(x)\}_\varepsilon$ be a family of standard symmetric mollifiers. For any $\varepsilon>0$ and $x\in\R^d$ we use $\rho_\varepsilon(x-\cdot)$ as test function, ten we get
$$
\begin{aligned}
      \int_{\R} V(t,y) \rho_\varepsilon(x-y) \, dy  = &\, - \int_{0}^{t}  \int_{\R} \big( b(y)  \partial_y V(s,y)  \big)  \rho_\varepsilon(x-y) \ dy ds
		\\[5pt]
    & -  \int_{0}^{t} \!\! \int_{\R} \partial_y V(s,y) \, \rho_\varepsilon(x-y)  \, dy \circ dB_s
\end{aligned}
$$

We set  $V_\varepsilon(t,x)= (V\ast \rho_\varepsilon)(x)$, $b_\varepsilon(x)= (b \ast \rho_\varepsilon)(x)$ and
$(bV)_\varepsilon(t,x)= (b.V\ast \rho_\varepsilon)(x)$. Then we deduce

$$
\begin{aligned}
    V_{\varepsilon}(t,x) + \int_{0}^{t} b_{\epsilon}(x)   \partial_x V_{\varepsilon}(s,x) \,  ds   +  \int_{0}^{t}   \partial_{x}  V_{\varepsilon}(s,x) \, \circ dB_s
				\\[5pt]
    & =         \int_{0}^{t} \big(\mathcal{R}_{\epsilon}(V,b) \big) (x,s) \,  ds ,
\end{aligned}
$$

\noindent where we denote
$ \mathcal{R}_{\epsilon}(V,b)  = b_\varepsilon \ \partial_x V_\varepsilon  -  (b\partial_x V)_\varepsilon  $.

\bigskip

{\it Step 3: Method of Characteristic.}
Applying the It\^o-Wentzell-Kunita formula   to $ V_{\varepsilon}(t,X_{t}^{\epsilon})$
, see Theorem 8.3 of \cite{Ku2}, we have

\[
   V_{\varepsilon}(t,X_{t}^{\epsilon})  = \int_{0}^{t} \big(\mathcal{R}_{\epsilon}(V,b) \big) (X_s^{\epsilon},s)  ds  .
\]

Then, considering that $X_{t}^{\epsilon}=X_{0,t}^{\epsilon}$ and $Y_{t}^{\epsilon}=Y_{0,t}^{\epsilon}=(X_{0,t}^{\epsilon})^{-1}$ 
we have that 

\[
   V_{\varepsilon}(t,x)  = \int_{0}^{t} \big(\mathcal{R}_{\epsilon}(V,b) \big) (X_{0,s}^{\epsilon}(Y_{0,t}^{\epsilon}),s)  ds=\int_{0}^{t} \big(\mathcal{R}_{\epsilon}(V,b) \big) (Y_{s,t}^{\epsilon},s)  ds  .
\]

Multiplying by the test functions $\varphi$ and integrating in $\R$ we get

\begin{equation}
   \int V_{\varepsilon}(t,x) \ \varphi(x) dx  =
	\int_{0}^{t}  \int \big(\mathcal{R}_{\epsilon}(V,b) \big) (Y_{s,t}^{\epsilon},s) \  \  \varphi(x) \ \, dx \   ds  .
\end{equation}

 \noindent We observe that

\begin{equation}
	\int_{0}^{t}  \int \big(\mathcal{R}_{\epsilon}(V,b) \big) (Y_{s,t}^{\epsilon},s) \ \varphi(x) \ \, dx \   ds =
	\int_{0}^{t}  \int \big(\mathcal{R}_{\epsilon}(V,b) \big) (x,s) \   JX_{s,t}^{\epsilon}  \varphi(X_{s,t}^{\epsilon}) \ \, dx \   ds .
\end{equation}

\bigskip

{\it Step 4: Convergence of the commutator.}   Now, we observe that $\mathcal{R}_{\epsilon}(V,b)$ converge to zero in
$L^{2}([0,T]\times \R )$. In fact,

\[
  (b \ \partial_x V)_{\varepsilon} \rightarrow b \ \partial_x V \ in \ L^{2}([0,T]\times \R ),
\]

and by the dominated convergence theorem we obtain

\[
b_{\epsilon}   \partial_x V_{\varepsilon} \rightarrow b \ \partial_x V \ in \ L^{2}([0,T]\times \R ).
\]

\bigskip

{\it Step 5: Conclusion.}  From step 3  we obtain

\begin{equation}\label{conv}
   \int V_{\varepsilon}(t,x) \ \varphi(x) dx  =
		\int_{0}^{t}  \int \big(\mathcal{R}_{\epsilon}(V,b) \big) (x,s) \   JX_{s,t}^{\epsilon}  \varphi(X_{s,t}^{\epsilon}) \ \, dx \   ds ,
\end{equation}

 Using  H\"older's inequality we have

\[
	\E \bigg|\int_{0}^{t}  \int \bigg(\mathcal{R}_{\epsilon}(V,b) \bigg) (x,s) \   JX_{s,t}^{\epsilon}  \varphi(X_{s,t}^{\epsilon}) \ \, dx \   ds \bigg|
	\]
	
	\[
	\leq   	 \bigg(\E \int_{0}^{t}  \int |\big(\mathcal{R}_{\epsilon}(V,b) \big) (x,s)|^{2} \  \, dx \   ds \bigg)^{\frac{1}{2}}
  \bigg(\E \int_{0}^{t}  \int  | JX_{s,t}^{\epsilon} \varphi(X_{s,t}^{\epsilon})|^{2} \ \, dx \ ds \bigg)^{\frac{1}{2}}
\]

\noindent  From step 4 result

\[
 \bigg(\E \int_{0}^{t}  \int |\big(\mathcal{R}_{\epsilon}(V,b) \big) (x,s)|^{2} \  \, dx \   ds \bigg)^{\frac{1}{2}}\rightarrow 0.
\]

From formula (\ref{est3}) we obtain 

\[
\bigg(\E \int_{0}^{t}  \int  | JX_{s,t}^{\epsilon} \varphi(X_{s,t}^{\epsilon})|^{2} \ \, dx \ ds \bigg)^{\frac{1}{2}}
\leq C \bigg(\E \int_{0}^{t}   |\varphi(x)|^{2} \ \, dx \ ds \bigg)^{\frac{1}{2}},
\]


Passing to the limit in equation (\ref{conv}) we conclude that $V=0$. Then we deduce that  $u=0$.

\end{proof}

\section{$H^{1}(\mathbb{R}^{d})$ Solutions.}

We  will be considered the divergence-free condition, that is
\begin{equation}
\label{DIVB2}
  \dive \,  b= 0
\end{equation}
(understood in the sense of distributions).
\bigskip

\begin{definition}\label{defisoluH}
 A stochastic process $u\in   L^{2}( \Omega\times[0, T],  H^{1}(\mathbb{R}^{d}) ) \cap  L^{\infty}( \Omega\times[0, T] \times \mathbb{R}^{d} )$ is called a  $ H^{1}$- weak solution of the Cauchy problem \eqref{trasport} when: For any $\varphi \in C_0^{\infty}(\R^d)$, the real valued process $\int  u(t,x)\varphi(x)  dx$ has a continuous modification which is an $\mathcal{F}_{t}$-semimartingale, and for all $t \in [0,T]$, we have $\mathbb{P}$-almost surely
\begin{equation} \label{DISTINTSTR}
\begin{aligned}
    \int_{\R^d} u(t,x) \varphi(x) dx = &\int_{\R^d} u_{0}(x) \varphi(x) \ dx
	-\int_{0}^{t} \!\! \int_{\R^d}  \partial_i u(s,x)  \cdot \, b^{i} (s,x) \varphi(x) \ dx ds
\\[5pt]
    &- \int_{0}^{t} \!\! \int_{\R^d} \partial_{i}  u(s,x) \ \varphi(x) \ dx \, {\circ}{dB^i_s} \, .
\end{aligned}
\end{equation}
\end{definition}


\begin{remark}  Analogously, as it was done in the remark \ref{lemmaito} we can write the problem (\ref{trasport}) in the
 It\^o form as follows, a  stochastic process $u   \in   L^{2}( \Omega\times[0, T],  H^{1}(\mathbb{R}^{d}) ) \cap  L^{\infty}( \Omega\times[0, T] \times \mathbb{R}^{d} )$
is  a $ H^{1}$- weak solution  of the SPDE (\ref{trasport}) iff for every test function $\varphi \in C_{0}^{\infty}(\mathbb{R}^{d})$, the process $\int u(t, x)\varphi(x) dx$ has a continuous modification which is a $\mathcal{F}_{t}$-semimartingale and satisfies the following It\^o's formulation

\begin{align*}
\int u(t,x) \varphi(x) dx  = & \int u_{0}(x) \varphi(x) \ dx \\[5pt]
& -\int_{0}^{t} \int b^{i}(s,x)\cdot \varphi(x) \partial_i  u(s,x) \ dx\, ds \\[5pt]
& - \int_{0}^{t} \int \varphi(x)  \partial_{i} u(s,x) \ dx  \ dB_{s}^{i}  \\[5pt]
& +\frac{1}{2}   \int_{0}^{t} \int  \Delta \,\varphi(x) u(s,x) \ dx\, ds.
\end{align*}

\end{remark}

\subsection{Existence.}

\begin{lemma}\label{lemmaexis1} We assume that $u_0\in H^{1}(\R^{d})  \cap L^{\infty}(\R^{d})$ and  conditions (\ref{LPSC}) and   
(\ref{DIVB2}). Then there exists    $H^{1}$- weak solution $u$ of the Cauchy problem \eqref{trasport}.
\end{lemma}

\begin{proof}

Let $\{\rho_\varepsilon\}_\varepsilon$ be a family of standard symmetric mollifiers. Consider a  nonnegative smooth cut-off function $\eta$ supported on the ball of radius 2 and such that $\eta=1$ on the ball of radius 1. For every $\varepsilon>0$ introduce the rescaled functions $\eta_\varepsilon (\cdot) = \eta(\varepsilon \cdot)$. Using these two families of functions we define the family of regularized coefficient as $b^{\epsilon}(t,x) = \eta_\varepsilon(x) \big( [ b(t,\cdot) \ast \rho_\varepsilon (\cdot) ] (x) \big) $. Similarly, we  define the family of regular approximations of the initial condition $u_0^\varepsilon (x) = \eta_\varepsilon(x) \big( [ u_0(\cdot) \ast \rho_\varepsilon (\cdot) ] (x)  \big) $.

Remark that any element $b^{\varepsilon}$, $u_0^\varepsilon$, $\varepsilon>0$ of the two families we have defined is smooth (in space) and compactly supported, therefore with bounded derivatives of all orders. Then, for any fixed $\varepsilon>0$, the classical theory of Kunita, see \cite{Ku} or \cite{Ku3}, provides the existence of an unique solution $u^{\varepsilon}$ to the regularized equation
\begin{equation}\label{trasport-reg}
 \left \{
\begin{aligned}
    &d u^\varepsilon (t, x,\omega) +  \nabla u^\varepsilon (t, x,\omega)  \cdot \big( b^\varepsilon (t, x)  dt +
 \circ d B_{t}(\omega) \big) = 0 \, ,
    \\[5pt]
    &u^\varepsilon \big|_{t=0}=  u_{0}^\varepsilon
\end{aligned}
\right .
\end{equation}
together with the representation formula
\begin{equation}\label{repr formula}
u^\varepsilon (t,x) = u_0^\varepsilon \big( (\phi_t^\varepsilon)^{-1} (x) \big)
\end{equation}
in terms of the (regularized) initial condition and the inverse flow $(\phi_t^\varepsilon)^{-1}$ associated to the equation of characteristics of \eqref{trasport-reg}, which reads
\begin{equation*}
d X_t = b^\varepsilon (t, X_t) \, dt + d B_t \, ,  \hspace{1cm}   X_0 = x \,.
\end{equation*}
Now, by Lemma 5 of  \cite{Fre1} we have that for every $p \geq 1$, there exists $C_{d,p,T}> 0$
such that
\begin{equation}\label{estima}
    \sup_{t \in [0,T]} \sup_{x \in \R^d} \mathbb{E}[| D(\phi_{t}^\varepsilon) |^p] \leq C_{d,p,T}, \quad \text{uniformly in $\epsilon> 0$}.
		\end{equation}
Then, we can use the random change of variables $(\phi_t^\varepsilon)^{-1}(x)  \mapsto x$ to obtain that
\begin{align}   \label{u_eps uniform}
\int_{\R^d} \E \big|  u^{\epsilon}(t,x) \big|^{2}  dx &=  \int_{\R^d} \E \big| u_{0}^{\epsilon}\big( (\phi_{t}^\varepsilon)^{-1} (x,\omega)\big) \big|^{2}  dx = \int_{\R^d}  \big|u_{0}^{\epsilon}(x) \big|^2 \, dx .   \nonumber  \\
\end{align}
Moreover, we have
\begin{align}   \label{u_eps uniform}
\int_{\R^d} \E \big| \nabla u^{\epsilon}(t,x) \big|^{2}  dx = & \int_{\R^d} \E \big| \nabla [ u_{0}^{\epsilon}\big( (\phi_{t}^\varepsilon)^{-1} (x,\omega)\big) ] \big|^{2}  dx \nonumber  \\
= & \int_{\R^d}  \big|\nabla u_{0}^{\epsilon}((\phi_{t}^\varepsilon)^{-1})  D(\phi_{t}^\varepsilon)^{-1} (x,\omega) \big|^2   dx \nonumber\\
= & \int_{\R^d}  \big|\nabla u_{0}^{\epsilon}(x) \big|^2   \E   |D(\phi_{t}^\varepsilon)^{-1} (\phi_{t}^\varepsilon ,\omega)|^{2} \, dx.\nonumber  \\
\end{align}
We observe that
\[
D(\phi_{t}^\varepsilon)^{-1} (\phi_{t}^\varepsilon ,\omega)= D^{-1}(\phi_{t}^\varepsilon),
\]
and
\[
 D^{-1}(\phi_{t}^\varepsilon)= Cof (D \phi_{t}^\varepsilon)^{T}
\]
where $Cof$ denoted the cofactor matrix of $D\phi_{t}^\varepsilon$.  By inequality (\ref{estima}) we deduce that
$ Cof (D\phi_{t}^\varepsilon)^{T} \in L^\infty \big(  \R^d\times[0,T], L^{2} (\Omega) \big) $. Thus we obtain that
\begin{align}   \label{u_eps uniform2}
\int_{\R^d} \E \big| \nabla u^{\epsilon}(t,x) \big|^{2}  dx & = \int_{\R^d}  \big|\nabla u_{0}^{\epsilon}(x) \big|^2   \E |  |D(\phi_{t}^\varepsilon)^{-1} (\phi_{t}^\varepsilon ,\omega)|^{2}| \, dx    \nonumber  \\
& \leq C  \int_{\R^d}  \big|\nabla u_{0}^{\epsilon}(x) \big|^2  \, dx.
\end{align}

\medskip

If $u^\varepsilon$ is a solution of \eqref{trasport-reg}, it is also a weak solution, which means that for any test function $\varphi \in C_0^\infty(\R^d)$, $u^\varepsilon$ satisfies
the following equation (written in It\^o form)
\begin{equation} \label{transintegralR2}
\begin{aligned}
    \int_{\R^d} u^\varepsilon(t,x) &\varphi(x) \, dx= \int_{\R^d} u^\varepsilon_{0}(x) \varphi(x) \, dx
   -\int_{0}^{t} \!\! \int_{\R^d}  \partial_i u^\varepsilon(s,x) \, b^{i,\varepsilon}(s,x)  \varphi(x) \, dx ds
\\[5pt]
    &-  \int_{0}^{t} \!\! \int_{\R^d}  \partial_{i} u^\varepsilon(s,x) \,  \varphi(x) \, dx \, dB^i_s
    + \frac{1}{2} \int_{0}^{t} \!\!\int_{\R^d} u^\varepsilon(s,x) \Delta \varphi(x) \, dx ds \, .
\end{aligned}
\end{equation}
To prove the existence of weak solutions to \eqref{trasport}  we  can extract subsequence (for simplicity the whole sequence), which converges   weakly to some $u$ in $L^{2}( \Omega\times[0, T],  H^{1}(\mathbb{R}^{d}) )$ and  weak star  in $ L^{\infty}( \Omega\times[0, T] \times \mathbb{R}^{d} )$. Then  via classical arguments, that is,  we can  pass to the limit in the above equation along this subsequence. This is done following the classical argument of \cite[Sect. II, Chapter 3]{Pardoux}, \cite[Theorem 15]{FGP2} and  \cite[Theorem 23]{Beck}.

\end{proof}

\subsection{Uniqueness.}

Before starting and proving the main theorem of this subsection, we shall introduce some further notation and the key lemma on commutators.

\bigskip
Let  $\{\rho_{\varepsilon} \}$ be a family of standard positive symmetric mollifiers. Given two functions $f:\R^d \mapsto \R^d$ and $g:\R^d \mapsto \R$, the commutator $\calR_\varepsilon(f,g)$ is defined as
\begin{equation}\label{def commut}
    \mathcal{R}_{\varepsilon}(f,g):= (f \cdot \nabla ) (\rho_{\epsilon}\ast g )- \rho_{\varepsilon}\ast  (f\cdot \nabla g ) \, .
  \end{equation}
The following lemma is due to C. Le Bris and P.-L. Lions \cite[Lemma 5.1.]{BrisLion}.

\begin{lemma} \label{conmuting} (C. Le Bris - P. L.Lions )
Let  $f \in  L^2_\loc(\R^d ) $,  $g \in H^{1}(\R^d) $. Then,
passing to the limit as $\varepsilon\rightarrow 0$
\[
    \mathcal{R}_{\varepsilon}(f,g) \rightarrow 0  \qquad in  \qquad L^1_\loc(\R^d) \, .
  \]
\end{lemma}
\bigskip

We can finally state our uniqueness result.
\bigskip

\begin{theorem}\label{uni}
Under the conditions   (\ref{LPSC}) and  (\ref{DIVB2}), uniqueness holds for  $H^{1}$- weak solutions of the Cauchy problem \eqref{trasport} in the following sense: if $u,v$ are $H^{1}$- weak solutions with the same initial data $u_{0}\in H^{1}(\mathbb{R}^{d})  \cap L^{\infty}(  \mathbb{R}^{d} )$, then  $u= v$ almost everywhere in $ \Omega  \times [0,T] \times \R^d $.
\end{theorem}

\begin{proof}

The proof is essentially based on energy-type estimates on $u$. However, to rigorously obtain it  two preliminary technical steps of regularization and localization are needed, where the above Lemma \ref{conmuting} will be used to deal with the commutators appearing in the regularization process.
\medskip

{\it Step 0: Set of solutions.} Remark that the set of  $H^{1}$- weak solutions is a linear subspace of $L^{2}( \Omega\times[0, T],  H^{1}(\mathbb{R}^{d}) )\cap   L^{\infty}( \Omega\times[0, T] \times \mathbb{R}^{d} )$, because the stochastic transport equation is linear, and the regularity conditions is a linear constraint. Therefore, it is enough to show that a $H^{1}$- weak solution $u$ with initial condition $u_0= 0$ vanishes identically.

\bigskip

{\it Step 1: Smoothing.}
Let $\{\rho_{\varepsilon}(x)\}_\varepsilon$ be a family of standard symmetric mollifiers. For any $\varepsilon>0$ and $x\in\R^d$ we can use $\rho_\varepsilon(x-\cdot)$ as test function, we get
$$
\begin{aligned}
      \int_{\R^d} u(t,y) \rho_\varepsilon(x-y) \, dy  = &\, - \int_{0}^{t}  \int_{\R^d} \big( b^{i}(s,y) \cdot \partial_i u(s,y)  \big)  \rho_\varepsilon(x-y) \ dy ds    \\[5pt]
    & - \frac{1}{2}\int_{0}^{t}   \int_{\R^d} \! \partial_i u(s,y) \,  \partial_i \, \rho_\varepsilon(x-y) \, dy  ds
		\\[5pt]
    & +  \int_{0}^{t} \!\! \int_{\R^d} u(s,y) \, \partial_i \rho_\varepsilon(x-y)  \, dy \, dB^i_s.
\end{aligned}
$$
Set $u_\varepsilon(t,x)= u(t,x) \ast_x \rho_\varepsilon(x)$. Using the definition \eqref{def commut} of the commutator $\big(\calR_{\epsilon}(f,g)\big) (s)$ with $f=b(s, \cdot)$ and $g=u(s, \cdot)$, we have for each $t \in [0,T]$
$$
\begin{aligned}
    u_{\varepsilon}(t,x) + \int_{0}^{t} b^{i}(s,x) \cdot  \partial_i u_{\varepsilon}(s,x) \,  ds   & -
			\frac{1}{2}\int_{0}^{t}   \Delta u_{\varepsilon}(s,x) \, ds   \\[5pt]
    & 	+ \int_{0}^{t}   \partial_{i}  u_{\varepsilon}(s,x) \, dB^i_s
				\\[5pt]
    & =         \int_{0}^{t} \big(\mathcal{R}_{\epsilon}(b,u) \big) (s) \,  ds  \, .
\end{aligned}
$$
By It\^o formula we have
\begin{align}\label{1000}
u_{\epsilon}^2(t,x) = & -\I_0^t  b^{i}(s,x)\cdot \partial_i u^2_{\epsilon}(s,x) ds -\int_{0}^{t} \int_{\R^d}\partial_{i} u_{\epsilon}^2(s,x) \ dx \, {dB^i_s} \nonumber\\
&\, +\frac{1}{2}\I_0^t \triangle u^2_{\epsilon}(s,x) dx + \I_0^t 2u_{\epsilon}(s,x)\mathcal{R}_{\epsilon} (b,u)(s,x) ds
\end{align}


\bigskip
{\it Step 2: Localization.} Consider a  nonnegative smooth cut-off function $\eta$ supported on the ball of radius 2 and such that  $\eta=1$ on the ball of radius 1. For any $R>0$ introduce the rescaled functions $\eta_R (\cdot) =  \eta(\frac{.}{R})$. Multiplying \eqref{1000} by $\eta_R$ and integrating over $ \R^d $ we have
\begin{align*}
\I_{\R^d} u_{\epsilon}^2(t,x)\eta_R(x) dx = & -\I_0^t\I_{\R^d} b^{i}(s,x)\cdot \partial_i u_{\epsilon}^2(s,x) \eta_R(x) dx ds \\[5pt]
                                 &  - \int_{0}^{t} \int_{\R^d} \partial_{i} u_{\epsilon}^2(s,x)\eta_R(x)  \ dx \,{dB^i_s} \nonumber \\[5pt]
                                 & +\frac{1}{2}\I_0^t\I_{\R^d} \triangle u^2_{\epsilon}(s,x) \eta_R(x)dx \,ds \\[5pt]
                                 & + \I_0^t \int_{\R^d} 2u_{\epsilon}(s,x)\mathcal{R}_{\epsilon} (b,u)(s,x) \eta_R(x) ds,
\end{align*}
which we rewrite as
\begin{align} \label{eq V_eps^2}
\I_{\R^d} u_{\epsilon}^2(t,x)\eta_R(x) dx = & \I_0^t\I_{\R^d} u_{\epsilon}^2(s,x)  b^{i}(s,x)\cdot \partial_i \eta_R(x) dx ds \nonumber\\[5pt]
                                 &  + \int_{0}^{t} \bigg(\int_{\R^d}  u_{\epsilon}^2(s,x) \partial_{i}\eta_R(x)  \ dx \bigg)\,{dB^i_s} \nonumber \\[5pt]
                                 & +\frac{1}{2}\I_0^t\I_{\R^d}  u^2_{\epsilon}(s,x) \triangle \eta_R(x)dx \,ds\nonumber \\[5pt]
                                 & + \I_0^t \int_{\R^d} 2u_{\epsilon}(s,x)\mathcal{R}_{\epsilon} (b,u)(s,x) \eta_R(x) ds.
\end{align}

\bigskip
{\it Step 4: Passage to the limit.} Finally, in this step we shall pass to the limit in $\varepsilon$ and $R$ to obtain uniqueness. We first take the limit $\varepsilon\rightarrow 0$ in the above equation \eqref{eq V_eps^2}. By standard properties of mollifiers $u_\varepsilon \to u$ strongly in $L^{2}\big([0,T]; H^1(\R^d)\big) $. Then  using  Lemma \ref{conmuting}, we obtain
 \begin{align}\label{eqen}
\I_{\R^d} u^2(t,x)\eta_R(x) dx = & \I_0^t\I_{\R^d} u^2(s,x)  b^{i}(s,x)\cdot \partial_i\eta_R(x) dx ds \nonumber\\[5pt]
                                 &   + \int_{0}^{t} \bigg(\int_{\R^d}  u^2(s,x) \partial_{i}\eta_R(x)  \ dx \bigg)\,{dB^i_s} \nonumber\\[5pt] & +\frac{1}{2}\I_0^t\I_{\R^d}  u^2(s,x) \triangle \eta_R(x)dx \,ds\,.
\end{align}
Now, we observe that
$$
\bigg|\int_{0}^{t}  \int_{\R^d}  u^{2}(s,x)  b^{i}(s,x) \cdot  \partial_i \eta_R(x) \,  dx   ds \bigg| \leq C \bigg|\int_{0}^{t}  \bigg(\int_{ R<|x|< 2R}    |b(s,x) |^{p} \,  dx  \bigg)^{q/p} ds\bigg|^{1/q}.
$$
Thus, taking the limit in (\ref{eqen}) as $R\rightarrow\infty$ we have
$$
\begin{aligned}
    \int_{\R^d}   u^{2}(t,x) \,  dx  = 0
	\end{aligned}
$$
Taking expectation and integrating on $[0,T]$ we have
$$
\begin{aligned}
    \I_{\Omega} \I_0^T \int_{\R^d}  u^{2}(t,x) \,  dx  \ dt \ \mathbb{P}(d\omega)=0\,.
	\end{aligned}
$$
Therefore, we conclude that  $u=0$ almost everywhere on $\Omega\times [0,T] \times \R^d$.
\end{proof}

\section{Appendix}

\begin{lemma}\label{lemma pe}
Assume $b\in C_c^{\infty}(\R)$ and that satisfies the hypothesis \ref{hyp1}. Then for  $T>0$  there are  constants $k_1=k_1(k,T)$ and $k_2=k_2(k,T)$ such that
\begin{align}\label{eq0}
\mathbb{E}\bigg[\bigg|\frac{d}{dx}X_t(x)\bigg|^{-2}\bigg] \leq  k_1  t^{-3/8} e^{k_2 x^2}  \,,
\end{align}
where $k_1=  \sqrt{c_1}\sqrt[4]{c_2}e^{35 T k^2}  $ and $k_2=2(k+99 T k^2)$ ($c_1$ and $c_2$ are defined below in the proof).
\end{lemma}

\begin{proof}
We consider the SDE associated to the vector field $b$ :
\begin{equation*}
d X_t = b (X_t) \, dt + d B_t \, ,  \hspace{1cm}   X_0 = x \,.
\end{equation*}

\noindent We denote
\begin{align*}
\mathcal{E}\bigg(\int_0^t b(X_u)dB_u\bigg)=\exp\bigg\{\int_0^t b(X_u)dB_u-\frac{1}{2} \int_0^{t} b^2(X_u)du \bigg\},
\end{align*}

 and
\begin{align*}
d Q(\omega)= \mathcal{E}\bigg(\int_0^t b(X_u)dB_u\bigg) d\mathbb{P}(\omega). 
\end{align*}
Using the Girsanov's theorem  we obtain that

\begin{align*}
\E\bigg[\bigg|\frac{d X_t}{dx}(x)\bigg|^{-2}\bigg ] & = \E_{Q}\bigg[\bigg|\frac{d X_t}{dx}(x)\bigg|^{-2}\bigg] \\
 & = \E\bigg[\exp\bigg\{-2\int_0^t b'(x+B_s) ds \bigg\} \mathcal{E}\bigg(\int_0^t b(x+B_s)dB_s\bigg)\bigg].
\end{align*}
Now, we proceed as in the proof of the Lemma 3.6 of \cite{nilssen}. Let $ b_1=-b$, then  we have
\begin{align*}
\E\bigg[\bigg|\frac{d X_t}{dx}(x)\bigg|^{-2}\bigg ] = \E\bigg[\exp\bigg\{2\int_0^t b_1'(x+B_s) ds \bigg\} \mathcal{E}\bigg(\int_0^t b(x+B_s)dB_s\bigg)\bigg]. 
\end{align*}

Applying the  Itô's formula to $\tilde{b}(z)=\int_{\infty}^z b_1(y)dy$  we get 

\begin{align*}
\tilde{b}(x+B_t)=\tilde{b}(x)+\int_0^t b_1(x+B_s) dB_s+\frac{1}{2} \int_0^t b_1'(x+B_s) ds\,.
\end{align*}

 By H\"older inequality we get

\begin{align}\label{eq1}
\E\bigg[\bigg|\frac{d X_t}{dx}(x)\bigg|^{-2}\bigg ] & = \E\bigg[\exp\bigg\{4(\tilde{b}(x+B_t) -\tilde{b}(x) -\int_0^t b_1(x+B_s) dB_s )\bigg\} \mathcal{E}\bigg(\int_0^t b(x+B_s)dB_s\bigg)\bigg] \nonumber\\
& \leq \|\exp\{4(\tilde{b}(x+B_t) -\tilde{b}(x) )\} \|_{L^2(\Omega)} \ \| \exp\bigg\{-4\int_0^t b_1(x+B_s) dB_s \bigg\}\times \nonumber\\
& \quad\quad\quad\quad\times \mathcal{E}\bigg(\int_0^t b(x+B_s)dB_s\bigg) \|_{L^2(\Omega)}\,.
\end{align}

 For the first term, we  have 
\begin{align*}
|\tilde{b}(x+B_t) -\tilde{b}(x)| & =|\int_0^1 b_1(x+\theta(B_t)) d\theta| \ |B_t| \\
& \leq \int_0^1 (k+k|x+\theta B_t|)d \theta |B_t| \\
& \leq k |B_t|+k|x||B_t|+\frac{k}{2}(B_t)^2 \\
& \leq \frac{k}{2}x^2+k |B_t|+k(B_t)^2 . 
\end{align*}
Thus we obtain 
\begin{align*}
\E[\exp\{8(\tilde{b}(x+B_t) -\tilde{b}(x) )\}] & \leq \E[\exp\{8(\frac{k}{2}x^2+k |B_t|+k(B_t)^2)\}] \\
& = e^{4kx^2} \E[\exp\{8(k |B_t|+k(B_t)^2)\}] \\
& = e^{4kx^2} \frac{1}{\sqrt{2\pi t}}\int_{\R} \ \exp\{8k(|z|+z^2)-\frac{z^2}{2 t}\} \ dz.
\end{align*}
Then , we conclude that 

\begin{align}\label{eq2}
\|\exp\{4(\tilde{b}(x+B_t) -\tilde{b}(x) )\} \|_{L^2(\Omega)} \leq  e^{2kx^2} t^{-1/4} \sqrt{c_1},
\end{align}

\noindent where
\begin{align*}
c_1=\frac{1}{\sqrt{2\pi }}\int_{\R} \exp\{8k(|z|+z^2)-\frac{z^2}{2 T}\} dz . 
\end{align*}

For the second term of \eqref{eq1} we have
\begin{align}\label{eq3}
\E & \bigg[\exp \bigg\{-8\int_0^t b_1(x+B_s) dB_s \bigg\} \mathcal{E}\bigg(\int_0^t b(x+B_s)dB_s\bigg)^2\bigg] \nonumber\\
& \E\bigg[\exp \bigg\{-8\int_0^t b_1(x+B_s) dB_s \bigg\} \exp\bigg\{2\int_0^t b(x+B_s)dB_s-\int_0^t b^2(x+B_s)ds\bigg\}\bigg] \nonumber\\
& =\E\bigg[\exp \bigg\{-10\int_0^t b_1(x+B_s) dB_s -\int_0^t b^2_1(x+B_s) ds\bigg\}\bigg] \nonumber\\
& = \E\bigg[\exp \bigg\{-10\int_0^t b_1(x+B_s) dB_s -\alpha \int_0^t b^2_1(x+B_s) ds\bigg\}\exp\bigg\{(\alpha-1)\int_0^t b^2_1(x+B_s) ds\bigg\}\bigg] \nonumber\\
& \leq \|\exp \bigg\{-10\int_0^t b_1(x+B_s) dB_s -\alpha \int_0^t b^2_1(x+B_s) ds\bigg\}\|_{L^2(\Omega)} \times \nonumber\\
& \quad\quad\times \|\exp\bigg\{(\alpha-1)\int_0^t b^2_1(x+B_s) ds\bigg\}\|_{L^2(\Omega)}. 
\end{align}
Now, we choose $\alpha=100$ because $\frac{1}{2}(-20 b_1(x+B_s))^2=2\alpha b_1^2(x+B_s)$. Then the process $\exp\{-20\int_0^t b_1(x+B_s) dB_s-200\int_0^t b_1^2(x+B_s)ds\}=\mathcal{E}\bigg(\int_0^t (-20b_1(x+B_s) dB_s)\bigg)$ is a martingale
with expectation equal to one. Then 

$$\|\exp \bigg\{-10\int_0^t b_1(x+B_s) dB_s -100 \int_0^t b^2_1(x+B_s) ds\bigg\}\|_{L^2(\Omega)}=1$$

From \eqref{cond3-1} we obtain that the second term of \eqref{eq3} is bounded by
\begin{align*}
\E\bigg[\exp\bigg\{2(\alpha-1) & \int_0^t b^2_1(x+B_s) ds\bigg\}\bigg] = \E\bigg[\exp\bigg\{198\int_0^t b^2_1(x+B_s) ds\bigg\}\bigg] \\
& \leq \E\bigg[\exp\bigg\{198\int_0^t k^2(1+|x+B_s|)^2 ds\bigg\}\bigg] \\
& \leq \E\bigg[\exp\bigg\{198t k^2(1+B_t^{*})^2 \bigg\}\bigg]\,,
\end{align*}
where $B_t^{*}=\displaystyle\sup_{s\leq t}|x+B_s|$. We define

$$Y_s=\exp\{99 t k^2(1+|x+B_s|)^2 \}.$$

 Then, by Doob's Maximal inequality we have
\begin{align*}
\E\bigg[ \exp\bigg\{198t k^2(1+B_t^{*})^2 \bigg\}\bigg] & =\E\bigg[\sup_{s\leq t}Y_s^2\bigg] \\
& \leq 4 \ \E[Y_t^2]=\E[\exp\{198 t k^2(1+|x+B_t|)^2\}] \\
& \leq 4 \ \E[\exp\{396 t k^2(1+(x+B_t)^2)\}] \\
& \leq 4 \ \E[\exp\{396 t k^2(1+2(x^2+B_t^2))\}] \\
& = 4e^{396 t k^2}e^{792 t k^2 x^2}\E[\exp\{792k^2 t B_t^2\}] \\
& = 4e^{396 t k^2}e^{792 t k^2 x^2} \frac{1}{\sqrt{2\pi t}}\int_{\R} \exp\{792 t k^2 z^2-\frac{z^2}{2 t}\} dz.
\end{align*}
Substituting in \eqref{eq3}  we get
\begin{align}\label{eq4}
\E \bigg[\exp \bigg\{-8\int_0^t b_1(x+B_s) dB_s \bigg\} \mathcal{E}\bigg(\int_0^t b(x+B_s)dB_s\bigg)^2\bigg] \leq e^{198 T k^2}e^{396 T k^2 x^2} t^{-1/4}\sqrt{c_2}\,,
\end{align}
where
$$c_2=\frac{4}{\sqrt{2\pi }}\int_{\R} \exp\{792 T k^2 z^2-\frac{z^2}{2 T}\} dz.$$
Therefore, replacing \eqref{eq2} and \eqref{eq4} in \eqref{eq1} we conclude
\begin{align*}
\E\bigg[\bigg|\frac{d X_t}{dx}(x)\bigg|^{-2}\bigg ]  &\leq   e^{2kx^2} t^{-1/4} \sqrt{c_1}e^{99 T k^2}e^{198 T k^2 x^2} t^{-1/8}\sqrt[4]{c_2}  \\
& = \sqrt{c_1}\sqrt[4]{c_2}e^{99 T k^2}  t^{-3/8}e^{2(k+99 T k^2) x^2}  \\
& = k_1  t^{-3/8} e^{k_2 x^2}\,,
\end{align*}
where $k_1=  \sqrt{c_1}\sqrt[4]{c_2}e^{99 T k^2}  $ and $k_2=2(k+99 T k^2)$.  This proves \eqref{eq0}.
\end{proof}

\section*{Acknowledgements}
    Christian Olivera C. O. is partially supported by  CNPq
through the grant 460713/2014-0 and FAPESP by the grants 2015/04723-2 and 2015/07278-0.


\end{document}